\documentstyle[amssymb,amsfonts,12pt]{amsart}

\newtheorem{theorem}{Theorem}[section]
\newtheorem{lemma}[theorem]{Lemma}
\newtheorem{corollary}[theorem]{Corollary}
\newtheorem{proposition}[theorem]{Proposition}
\newtheorem{definition}[theorem]{Definition}

\theoremstyle{remark}
\newtheorem{remark}[theorem]{Remark}
\newtheorem*{ack*}{Acknowledgment}

\def\p{\vec{p}}

\def\F{{\mathcal F}}
\def\R{{\mathbb R}}

\def\C{{\mathbb C}}
\def\Q{{\mathcal T}}\def\B{{\mathcal B}}

\def\P{{\mathcal P}}
\def\M{{\mathcal M}}\def\G{{\bf Gr}}

\def\A{{\mathcal N}}

\def\dim{{\operatorname{dim}}}\def\Col{{\textbf{Col}}}\def\Reach{{\operatorname{Reach}}}
\def\MLR{{\textbf{MLR}}}
\def\supp{{\operatorname{supp}}}\def\det{{\operatorname{det}}}
\def\bas{\begin{align*}}
\def\eas{\end{align*}}
\def\bi{\begin{itemize}}
\def\ei{\end{itemize}}
\newenvironment{proof}{\noindent {\bf Proof} }{\endprf\par}
\def \endprf{\hfill  {\vrule height6pt width6pt depth0pt}\medskip}
\def\emph#1{{\it #1}}

\begin{document}
\author{Jean Bourgain}
\address{School of Mathematics, Institute for Advanced Study, Princeton, NJ 08540}
\email{bourgain@@math.ias.edu}
\author{Ciprian Demeter}
\address{Department of Mathematics, Indiana University, 831 East 3rd St., Bloomington IN 47405}
\email{demeterc@@indiana.edu}

\title[Optimal mean value estimates for the quadratic Weyl sums]{Optimal mean value estimates for the quadratic Weyl sums in two dimensions}

\begin{abstract}
We use decoupling theory to prove a sharp (up to $N^\epsilon$ losses) estimate for Vinogradov's mean value theorem in two dimensions.
\end{abstract}
\maketitle
\section{Introduction}
\medskip

Let $\M$ denote the manifold
$$\M=\{(s,t,s^2,t^2,st):\;0\le s,t\le 1\}.$$
For each square $S\subset [0,1]^2$ and each $g:S\to\C$ define the extension operator
$$E_Sg(x_1,\ldots,x_5)=\int_Sg(s,t)e(x_1s+x_2t+x_3s^2+x_4t^2+x_5st)dsdt.$$
Here and throughout the rest of the paper we will  write $$e(z)=e^{2\pi i z},\;z\in\R.$$
For a positive weight $v:\R^5\to[0,\infty)$ we define
$$\|f\|_{L^p(v)}=(\int_{\R^5}|f(x)|^pv(x)dx)^{1/p}.$$
Also, for each ball $B$ in $\R^5$ centered at $c(B)$ and with radius $R$, $w_{B}$ will denote the weight
$$w_{B}(x)= \frac{1}{(1+\frac{|x-c(B)|}{R})^{100}}.$$
Throughout the paper, $B_R$ will denote an arbitrary ball in $\R^5$ with radius $R$. Our main result is the following decoupling theorem for $\M$.
\begin{theorem}
\label{tfek6}
For each $p\ge 2$,  $g:[0,1]^2\to\C$ and each ball $B_N\subset\R^5$ with radius $N\ge 1$ we have
$$
\|E_{[0,1]^2}g\|_{L^p(w_{B_{N}})}\le $$
\begin{equation}
\label{fe3}
\le D(N,p)(\sum_{\Delta\subset [0,1]^2\atop{l(\Delta)=N^{-1/2}}}\|E_\Delta g\|_{L^p(w_{B_{N}})}^p)^{1/p},
\end{equation}
 where the sum  is over a finitely overlapping cover of $[0,1]^2$ with squares $\Delta$ of side length $l(\Delta)=N^{-1/2}$, and for each $\epsilon>0$ we have
$$D(N,p)\lesssim_{\epsilon,p} N^{\frac12-\frac1p+\epsilon},\,2\le p\le 8,$$
$$D(N,p)\lesssim_{\epsilon,p} N^{1-\frac5{p}+\epsilon},\,p\ge 8.$$
\end{theorem}
In the following, we may and will implicitly assume that $N=2^m$ for some positive integer $m$, and that the squares $\Delta$ are dyadic and partition $[0,1]^2$.

A standard computation with $g=1_{[0,1]^2}$ reveals that  Theorem \ref{tfek6} is essentially sharp, more precisely
\begin{equation}
\label{fe30}
D(N,p)\gtrsim N^{\frac12-\frac1p}\text{ for }2\le p\le 8,\;\;D(N,p)\gtrsim N^{1-\frac5{p}}\text{ for }8\le p\le \infty.
\end{equation}
For future use, we record the following trivial upper bound that follows from the Cauchy--Schwartz inequality
\begin{equation}
\label{fe3008}
D(N,p)\lesssim N^{1-\frac1p},\text{ for } p\ge 1.
\end{equation}
We will prove that $D(N,8)\lesssim_\epsilon N^{\frac3{8}+\epsilon}$. The estimates for other $p$  will follow by interpolation with the trivial $p=2$ and $p=\infty$ results.
\bigskip

Theorem \ref{tfek6} is part of a program that has been initiated by the authors in \cite{BD3},  where the sharp decoupling theory has been completed for hyper-surfaces with definite second fundamental form, and also for the cone. The decoupling theory has since proved to be a very successful tool for a wide variety of problems in number theory that involve exponential sums. See  \cite{Bo}, \cite{Bo6}, \cite{BW}, \cite{BD4}, \cite{BD5}. This paper is no exception from the rule. Theorem \ref{tfek6} is in part motivated by its application to solving the Vinogradov-type mean value conjecture for quadratic systems in two dimensions, as explained in the next section. Perhaps surprisingly, our Fourier analytic approach eliminates any appeal to number theory. The methodology we develop here is in principle applicable to address the similar question in all dimensions, under the quadraticity assumption. We have decided not to pursue this general case here.

\bigskip

The proof of Theorem \ref{tfek6} will follow a strategy similar to the one from \cite{BD5}. At the heart of the argument lies the interplay between linear and multilinear decoupling, facilitated by the Bourgain--Guth induction on scales. Running this machinery produces two types of contributions, a transverse one and a non-transverse one. To control the transverse term we need to prove a $10-$linear restriction theorem for a specific two dimensional manifold in $\R^5$. Defining transversality in a manner that makes it easy to check and achieve in our application, turns out to be a rather delicate manner. In the attempt to simplify the discussion, we often run non-quantitative arguments that  rely instead on compactness. For example, in line with our previous related papers, we never care about the quantitative dependence on transversality of the bound in the multilinear restriction inequality. These considerations occupy sections \ref{Bra}, \ref{Trans} and \ref{se:multi}.

The non-transverse contribution is dominated using a trivial form of decoupling. But to make this efficient, we have to make sure that there are not too many transverse terms contributing to the sum. This is achieved in Section \ref{Comb} via some geometric combinatorics that we find of independent interest.

\begin{ack*}
The second author  would like to thank Mariusz Mirek and Lillian Pierce for drawing his attention to the Vinogradov mean value theorem in higher dimensions.
\end{ack*}

\bigskip

\section{Number theoretical consequences}
For each integer $s\ge 1$, denote by  $J_{s,2,2}(N)$ the number of integral solutions for the following system of simultaneous Diophantine equations
$$X_1+\ldots+X_s=X_{s+1}+\ldots+X_{2s},$$
$$Y_1+\ldots+Y_s=Y_{s+1}+\ldots+Y_{2s},$$
$$X_1^2+\ldots+X_s^2=X_{s+1}^2+\ldots+X_{2s}^2,$$
$$Y_1^2+\ldots+Y_s^2=Y_{s+1}^2+\ldots+Y_{2s}^2,$$
$$X_1Y_1+\ldots+X_sY_s=X_{s+1}Y_{s+1}+\ldots+X_{2s}Y_{2s},$$
with $1\le X_i,Y_j\le N$.

It was  conjectured in \cite{PPW} (see the top of page 1965, with $k=d=2$) that for $s\ge 1$
$$J_{s,2,2}(N)\lesssim_{\epsilon,s} N^{\epsilon}(N^{2s}+N^{4s-8}).$$
This is the quadratic case of the two dimensional Vinogradov mean value theorem.

Theorem 1.1 in \cite{PPW} established this inequality for $s\ge 15$. Here we will prove that this holds  in the whole range $s\ge 1$. Our approach will in fact prove a much more general result, see Corollary \ref{cfek4} below. We start with the following rather immediate consequence of our Theorem \ref{tfek6}.

\begin{theorem}
For each $1\le i\le N$, let $t_i,s_i$ be two points in $(\frac{i-1}{N},\frac{i}{N}]$. Then for each $R\gtrsim N^{2}\ge 1$, each ball $B_R$ with radius $R$ in $\R^5$, each $a_{i,j}\in\C$ and each $p\ge 2$  we have
$$(\frac1{|B_R|}\int_{B_R}|\sum_{i=1}^N\sum_{j=1}^Na_{i,j}e(x_1s_i+x_2t_j+x_3s_i^2+x_4t_j^2+x_5s_it_j)|^{p}dx_1\ldots dx_5)^{\frac1p}\lesssim$$
\begin{equation}
\label{fek19}
\lesssim D(N^2,p)\|a_{i,j}\|_{l^p(\{1,\ldots,N\}^2)},
\end{equation}
and the implicit constant does not depend on $N$, $R$ and $a_{i,j}$.
\end{theorem}
\begin{proof}
Given $B_R$, let $\B$ be a finitely overlapping cover of $B_R$ with balls $B_{N^2}$. An elementary computation shows that
\begin{equation}
\label{fek20}
\sum_{B_{N^2}\in\B}w_{B_{N^2}}\lesssim w_{B_R},
\end{equation}
with the implicit constant independent of $N,R$. Invoking  Theorem \ref{tfek6} for each $B_{N^2}\in\B$, then summing up and using
\eqref{fek20} we obtain
$$
\|E_{[0,1]^2}g\|_{L^p(B_R)}\lesssim $$
$$
\lesssim D(N^2,p)(\sum_{\Delta\subset [0,1]^2\atop{l(\Delta)=N^{-1}}}\|E_\Delta g\|_{L^p(w_{B_{R}})}^p)^{1/p}.
$$
Use this inequality with $$g=\frac1{\tau^2}\sum_{i=1}^N\sum_{j=1}^Na_{i,j}1_{B_{i,j,\tau}},$$
where $B_{i,j,\tau}$ is the ball in $\R^2$ centered at $(s_i,t_j)$ with radius $\tau.$ Then let $\tau$ go to $0$.

\end{proof}
For each $1\le i\le N$ consider some real  numbers $i-1< \tilde{X}_i,\tilde{Y}_i\le i$. We do not insist that $\tilde{X}_i,\tilde{Y}_i$ be integers. Let $S_X=\{\tilde{X}_1,\ldots,\tilde{X}_N\}$ and $S_Y=\{\tilde{Y}_1,\ldots,\tilde{Y}_N\}$. For each $s\ge 1$, denote by  $\tilde{J}_{s,2,2}(S_X,S_Y)$ the number of  solutions  of the following system of inequalities
$$|X_1+\ldots+X_s-(X_{s+1}+\ldots+X_{2s})|\le \frac1{N},$$
$$|Y_1+\ldots+Y_s-(Y_{s+1}+\ldots+Y_{2s})|\le \frac1{N},$$
$$|X_1^2+\ldots+X_s^2-(X_{s+1}^2+\ldots+X_{2s}^2)|\le 1,$$
$$|Y_1^2+\ldots+Y_s^2-(Y_{s+1}^2+\ldots+Y_{2s}^2)|\le 1,$$
$$|X_1Y_1+\ldots+X_sY_s-(X_{s+1}Y_{s+1}+\ldots+X_{2s}Y_{2s})|\le 1,$$
with $X_i\in S_X, Y_j\in S_Y$.
\begin{corollary}
\label{cfek4}
For each integer $s\ge 1$ and each $S_X,S_Y$ as above we have that
$$\tilde{J}_{s,2,2}(S_X,S_Y)\lesssim_{\epsilon,s}N^{\epsilon}(N^{2s}+N^{4s-8}),$$
where the implicit constant does not depend on $S_X,S_Y$.
\end{corollary}
\begin{proof}Let $\phi:\R^5\to [0,\infty)$ be a positive Schwartz function with  positive Fourier transform  satisfying $\widehat{\phi}(\xi)\ge1$ for $|\xi|\lesssim 1$.
Define $\phi_{N}(x)=\phi(\frac{x}N)$. Using the Schwartz decay, \eqref{fek19} with $a_{i,j}=1$ implies that for each $s\ge 1$
$$(\frac1{|B_{N^2}|}\int_{\R^5}\phi_{N^2}(x_1,\ldots,x_5)|\sum_{i=1}^N\sum_{j=1}^Ne(x_1s_i+x_2t_j+x_3s_i^2+x_4t_j^2+x_5s_it_j)|^{2s}dx_1\ldots dx_5)^{\frac1{2s}}\lesssim$$
\begin{equation}
\label{fek21}
\lesssim D(N^2,2s)N^{\frac1s},
\end{equation}
whenever $s_i,t_i\in [\frac{i-1}{N},\frac{i}{N})$. Apply \eqref{fek21} to $s_i=\frac{\tilde{X}_i}{N}$ and $t_j=\frac{\tilde{Y}_j}{N}$.
Let now $$\phi_{N, 1}(x_1,\ldots,x_5)=\phi(\frac{x_1}N,\frac{x_2}N, {x_3}, {x_4},{x_5}).$$

After making a change of variables and expanding the product,  the term
$$\int_{\R^5}\phi_{N^2}(x_1,\ldots,x_5)|\sum_{i=1}^N\sum_{j=1}^Ne(x_1s_i+x_2t_j+x_3s_i^2+x_4t_j^2+x_5s_it_j)|^{2s}dx_1\ldots dx_5$$
can be written as  the sum over all $X_i\in S_X,Y_j\in S_Y$ of
$$N^8\int_{\R^5}\phi_{N, 1}(x_1,\ldots,x_5)e(x_1Z_1+x_2Z_2+x_3Z_3+x_4Z_4+x_5Z_5)dx_1\ldots dx_5,$$
where
$$Z_1=X_1+\ldots+X_s-(X_{s+1}+\ldots+X_{2s}),$$
$$Z_2=Y_1+\ldots+Y_s-(Y_{s+1}+\ldots+Y_{2s}),$$
$$Z_3=X_1^2+\ldots+X_s^2-(X_{s+1}^2+\ldots+X_{2s}^2),$$
$$Z_4=Y_1^2+\ldots+Y_s^2-(Y_{s+1}^2+\ldots+Y_{2s}^2),$$
$$Z_5=X_1Y_1+\ldots+X_sY_s-(X_{s+1}Y_{s+1}+\ldots+X_{2s}Y_{2s}).$$
Each such term is equal to
$$N^{10}\widehat{\phi}(NZ_1,NZ_2,Z_3,Z_4, Z_5).$$
Recall that this is always positive, and in fact greater than $N^{10}$ at least $\tilde{J}_{s,2,2}(S_X,S_Y)$ times. Going back to \eqref{fek21}, it follows by invoking Theorem \ref{tfek6} that
$$\tilde{J}_{s,2,2}(S_X,S_Y)\lesssim D(N^2,2s)N^{\frac1s}\lesssim_{\epsilon,s} N^{\epsilon}(N^{2s}+N^{4s-8}).$$

\end{proof}
\bigskip

\section{A Brascamp--Lieb inequality}
\label{Bra}

For $1\le j\le m$, let $V_j$ be $n_j-$dimensional affine subspaces of $\R^n$ and let $l_j:\R^n\to V_j$ be surjective affine transformations.  Define the multilinear functional
$$\Lambda(f_1,\ldots,f_m)=\int_{\R^n}\prod_{j=1}^mf_j(l_j(x))dx$$
for $f_j:V_j\to\C$. Each $V_j$ will be equipped with the $n_j-$ dimensional Lebesgue measure. We recall the following theorem from \cite{BCCT}.
\begin{theorem}
\label{BCCT}
Given a vector $\p=(p_1,\ldots,p_m)$ with $p_j\ge 1$, we have that
\begin{equation}
\label{fek17}
\sup_{f_j\in L^{p_j}({V_j})}\frac{|\Lambda(f_1,\ldots,f_m)|}{\prod_{j=1}^m\|f_j\|_{L^{p_j}}}<\infty
\end{equation}
if and only if
\begin{equation}
\label{fek1}
n=\sum_{j=1}^m\frac{n_j}{p_j}
\end{equation}
and the following transversality condition is satisfied
\begin{equation}
\label{fek2}
\dim(V)\le \sum_{j=1}^m\frac{\dim(l_j(V))}{p_j}, \text{ for every subspace }V\subset \R^n.
\end{equation}
\end{theorem}
When all $p_j$ are equal to some $p$, an equivalent way to write \eqref{fek17} is
\begin{equation}
\label{fek18}
\sup_{g_j\in L^{2}({V_j})}\frac{\|(\prod_{j=1}^m g_j\circ l_j)^{\frac1m}\|_{L^{q}}}{(\prod_{j=1}^m\|g_j\|_{L^{2}})^{\frac1m}}<\infty,
\end{equation}
where $q=\frac{2n}{\sum_{j=1}^m n_j}$.

We will be interested in the special case when $V_j$ are linear subspaces, $l_j=\pi_j$ are orthogonal projections, $n=5$, $m=10$, $p_j=4$ and $n_j=2$. Note that \eqref{fek1} is satisfied in this case.  For future use, we reformulate the theorem in this case.
\begin{theorem}
\label{BCCT1}
The quantity
$$\sup_{g_j\in L^{2}({V_j})}\frac{\|(\prod_{j=1}^{10}(g_j\circ \pi_j))^{\frac1{10}}\|_{L^5(\R^5)}}{(\prod_{j=1}^{10}\|g_j\|_{L^{2}(V_j)})^{\frac1{10}}}$$
is finite if and only if
\begin{equation}
\label{fek3}
\dim(V)\le \frac14\sum_{j=1}^{10}{\dim(\pi_j(V))}, \text{ for every linear subspace }V\subset \R^5.
\end{equation}
\end{theorem}
We chose to work with $10-$ linearity because the $V_j$ to which we will apply the Brascamp--Lieb inequality satisfy the corresponding transversality assumption \eqref{fek3}. This will be proved in Section \ref{Trans}. With a bit of additional effort, the number 10 can be lowered to a smaller one. We have made no attempt to discover what this number is, as this would have no effect on the results of the paper. It is worth mentioning however that the $V_j$ that we work with are slightly less transverse than the generic ones.

Remark \ref{rek456} will show the relevance of the space $L^5$ from Theorem \ref{BCCT1}.

\bigskip

\section{Transversality}
\label{Trans}
In this section we introduce a quantitative form of transversality suited for our purposes and will prove a "uniform version`` of the Brascamp--Lieb inequality. This will be a first step towards proving the $10-$ linear restriction Theorem \ref{tfek4} in Section \ref{se:multi}.

Given two vectors $u=(u_1,\ldots,u_5),v=(v_1,\ldots,v_5)$ in $\R^5$, define the quadratic function
$$Q_{u,v}(x,y)=2x^2(v_3u_5-v_5u_3)+2y^2(v_5u_4-v_4u_5)+4xy(v_3u_4-v_4u_3)+$$
$$+2x(v_3u_2-v_2u_3)+2y(v_5u_2-v_2u_5)+v_1u_2-v_2u_1=0.$$
\begin{definition}
\label{dfek1}
A collection consisting of ten sets $S_1,\ldots,S_{10}\subset [0,1]^2$ is said to be $\nu-$transverse if the following two requirements are satisfied:
\medskip

(i) for each $i\not= j\not= k\in\{1,2,\ldots,10\}$ and each $(x_i,y_i)\in S_i$, $(x_j,y_j)\in S_j$ and $(x_k,y_k)\in S_k$ there is a permutation $\pi:\{i,j,k\}\to \{i,j,k\}$ such that
\begin{equation}
\label{fek4}
\big|(y_{\pi(j)}-y_{\pi(i)})[(x_{\pi(j)}-x_{\pi(i)})(y_{\pi(k)}-y_{\pi(i)})-(x_{\pi(k)}-x_{\pi(i)})(y_{\pi(j)}-y_{\pi(i)})]\big|\ge \nu,
\end{equation}
\medskip

(ii) for each $(x_j,y_j)\in S_{i_j}$ with $1\le i_1\not=\ldots\not=i_{5}\le 10$ we have
\begin{equation}
\label{fek5}
\inf_{u,v,w}\max_j(|Q_{u,v}(x_j,y_j)|+|Q_{u,w}(x_j,y_j)|)\ge \nu,
\end{equation}
where the infimum is taken over all orthonormal triples $(u,v,w)\in\R^5\times \R^5\times \R^5$.

\end{definition}
It is immediate that transverse sets are pairwise disjoint. Requirement (i) is essentially about the fact that three points in three different sets $S_i$ do not come "close`` to sitting on a line. Requirement (ii) says that any five points from five different sets do not sit "close`` to two quadratic curves of the type $Q_{u,v}$ and $Q_{u,w}$. Note that since $u,v,w$ are linearly independent, $Q_{u,v}=0$ and $Q_{u,w}=0$ are always distinct curves, and thus, their intersection is either a line or a finite set with   at most four points. This shows that given (i), the requirement (ii) above is always satisfied if the inequality $\ge \nu$ is replaced with $>0$.

The relevance of this definition for Theorem \ref{BCCT1} is presented in the following result. It may help to realize that the tangent plane at the point $(x,y,x^2,y^2,xy)\in\M$ is spanned by the vectors $(1,0,2x,0,y)$ and $(0,1,0,2y,x)$.
\begin{proposition}
\label{pfek1}
Assume we have ten points $(x_j,y_j)\in [0,1]^2$ such that the sets $S_j=\{(x_j,y_j)\}$ are $\nu-$transverse for some $\nu>0$. Then the ten planes $V_j,\;1\le j\le 10$  spanned by the vectors $n_j=(1,0,2x_j,0,y_j)$ and $m_j=(0,1,0,2y_j,x_j)$ satisfy requirement \eqref{fek3}.
\end{proposition}
\begin{proof}
We start with the easy observation that the ten $V_j$ are distinct. Indeed, note that the rank of the matrix
$$\begin{bmatrix}1&0&2x_i&0&y_i\\1&0&2x_j&0&y_j\\0&1&0&2y_i&x_i\\0&1&0&2y_j&x_j\end{bmatrix}$$
is at least three if $(x_i,y_i)\not=(x_j,y_j)$.

It suffices to check \eqref{fek3} for linear subspaces $V$ with dimension between 1 and 4, as the case of dimension 5 is trivial.

The next observation is that a one dimensional subspace can not be orthogonal to three distinct $V_j$. If this were to be the case with $V_i,V_j,V_k$, then these three planes would be forced to belong to a hyperplane. This in turn would force (for example)
$$\det \begin{bmatrix}1&0&2x_i&0&y_i\\1&0&2x_j&0&y_j\\1&0&2x_k&0&y_k\\0&1&0&2y_i&x_i\\0&1&0&2y_j&x_j\end{bmatrix}$$
to be zero, contradicting \eqref{fek4}. This observation shows that \eqref{fek3} is satisfied if $\dim(V)\le 2$, as $\dim(\pi_j(V))\ge 1$ for at least eight values of $j$.

Consider now the case of $V$ with $\dim(V)=3$ and with orthonormal basis $u,v,w$. We will argue that there are at least six $V_j$ with $\dim(\pi_j(V))=2$.  This immediately implies \eqref{fek3}.
Assume for contradiction that $\dim(\pi_j(V))\le 1$ for five values of $j$. By the rank-nullity theorem, we have that $V$ contains a two dimensional subspace $W_j$ orthogonal to $V_j$. This is the same as saying that the rank of the matrix
$$\begin{bmatrix}n_j\cdot u&n_j\cdot v& n_j\cdot w\\ m_j\cdot u&m_j\cdot v& m_j\cdot w\end{bmatrix}.$$
is at most one. In particular,
$$\det\begin{bmatrix}n_j\cdot u&n_j\cdot v\\ m_j\cdot u&m_j\cdot v\end{bmatrix}=0=\det\begin{bmatrix}n_j\cdot u&n_j\cdot w\\ m_j\cdot u&m_j\cdot w\end{bmatrix}.$$
This amounts to  $Q_{u,v}(x_j,y_j)=Q_{u,w}(x_j,y_j)=0$, for five values of $j$.
Note however that this contradicts \eqref{fek5}.

The last case that deserves analysis is $\dim(V)=4$. We will show that there can be at most two $V_j$ with $\dim(\pi_j(V))\le 1$, and thus \eqref{fek3} will again be satisfied. The rank-nullity theorem implies that $\dim(\pi_j(V))\le 1$ is equivalent with the existence of a three dimensional subspace $W_j$ of $V$ orthogonal to $V_j$. If this happened for three values of $j$, there would exist a one dimensional subspace orthogonal to these $V_j$. This in turn would force the three $V_j$ to belong to a hyperplane, a scenario that has been ruled out earlier.

\end{proof}

We can now prove the following "uniform version`` of the Brascamp--Lieb inequality.

\begin{theorem}
\label{tfek2}
Assume the planes $V_j,\;1\le j\le 10$ through the origin are spanned by the vectors $(1,0,2x_j,0,y_j)$ and $(0,1,0,2y_j,x_j)$ with $(x_j,y_j)\in [0,1]^2$ satisfying \eqref{fek4} and \eqref{fek5} for some $\nu>0$. Denote as before by $\pi_j$ their associated orthogonal projections. Then there exists a constant $\Theta_\nu<\infty$ depending only on $\nu$ such that
$$\sup_{g_j\in L^{2}({V_j})}\frac{\|(\prod_{j=1}^{10}(g_j\circ \pi_j))^{\frac1{10}}\|_{L^5(\R^5)}}{(\prod_{j=1}^{10}\|g_j\|_{L^{2}(V_j)})^{\frac1{10}}}\le \Theta_\nu.$$
\end{theorem}
\begin{proof}
The proof will rely on a few well-known or easy to check observations.
The Grassmannian $\G(2,\R^5)$ is the collection of all (two dimensional) planes containing the origin in $\R^5$. It is a compact metric space when equipped with the metric
$$d_{\G(2,\R^5)}(X,Y)=\|P_X-P_Y\|,$$
where $P_X,P_Y$ are the associated projections, and their difference is measured in the operator norm. The function $$F:\G(2,\R^5)^{10}\to\C^*$$ defined by
$$F(V_1,\ldots,V_{10})=\sup_{g_j\in L^{2}({V_j})}\frac{\|(\prod_{j=1}^{10}(g_j\circ \pi_j))^{\frac1{10}}\|_{L^5(\R^5)}}{(\prod_{j=1}^{10}\|g_j\|_{L^{2}(V_j)})^{\frac1{10}}}$$
is continuous, when the Riemann sphere $\C^*$ is equipped with the spherical metric.

The collection  $C$ of all ten-tuples $(V_1,\ldots,V_{10})$ satisfying our  hypothesis is closed in $\G(2,\R^5)^{10}$ (with respect to the product topology), hence compact. Assume for contradiction that the conclusion of the theorem fails. Using compactness and the continuity of $F$, it follows that $F(V_1,\ldots,V_{10})=\infty$ for some $(V_1,\ldots,V_{10})\in C$. This however is impossible, due to Theorem \ref{BCCT1} and Proposition \ref{pfek1}.

\end{proof}

\section{The $10-$linear restriction theorem}
\label{se:multi}

 For each $S\subset [0,1]^2$ and each real number $N\ge 1$, let $\A_{S,\frac1N}$ be the $N^{-1}-$neighborhood of
$$\M_S:=\{(x,y,x^2,y^2,xy):(x,y)\in S\}.$$

The key result proved in this section is the following $10-$linear restriction theorem. It is a close relative of the multilinear restriction theorem of Bennett, Carbery and Tao \cite{BCT}. The main difference is that while their theorem applies to hyper-surfaces, our result below is for the manifold $\M$ with co-dimension three.
\begin{theorem}
\label{tfek4}For each $0<\nu\ll 1$, each $\nu-$transverse squares $S_1,\ldots,S_{10}\subset [0,1]^2$, each $f_j:\A_{S_j,\frac1N}\to\C$, each $\epsilon>0$ and  each ball $B_N$ in $\R^5$ with radius $N\ge 1$  we have
$$\|(\prod_{j=1}^{10}\widehat{f}_j)^{\frac1{10}}\|_{L^5(B_N)}\lesssim_{\epsilon,\nu} N^{\epsilon-\frac32}(\prod_{j=1}^{10}\|f_j\|_{L^2(\A_{S_j,\frac1N})})^{\frac1{10}}.$$
\end{theorem}

\begin{remark}
\label{rek456}
It is rather immediate that
$$\|(\prod_{j=1}^{10}\widehat{f}_j)^{\frac1{10}}\|_{L^{\infty}(B_N)}\le (\prod_{j=1}^{10}\|\widehat{f}_j\|_{L^{\infty}(\R^5)})^{\frac1{10}}\lesssim  N^{-\frac32}(\prod_{j=1}^{10}\|f_j\|_{L^2(\A_{S_j,\frac1N})})^{\frac1{10}}.$$
When combined with Theorem \ref{tfek4}, this shows that
\begin{equation}
\label{fio vnhtyrgwrugj8btm8-8=-2r93]xe=-34.or,0-cvr9yt0-}
\|(\prod_{j=1}^{10}\widehat{f}_j)^{\frac1{10}}\|_{L^p(B_N)}\lesssim_{\epsilon,\nu} N^{\epsilon-\frac32}(\prod_{j=1}^{10}\|f_j\|_{L^2(\A_{S_j,\frac1N})})^{\frac1{10}}
\end{equation}
holds for each $p\ge 5$. The fact that it holds precisely for $p=5$ will be crucial for achieving the sharp range in Theorem \ref{tfek6}. On the other hand,  $p$ can not be lowered below 5 in \eqref{fio vnhtyrgwrugj8btm8-8=-2r93]xe=-34.or,0-cvr9yt0-}. Indeed, apply \eqref{fio vnhtyrgwrugj8btm8-8=-2r93]xe=-34.or,0-cvr9yt0-} with
$\widehat{f_j}=\phi_{T_j}$, where $\phi_{T_j}$ is a single wave-packet as in \eqref{fek9}. We can arrange  the intersection of the plates $T_j$ to contain a ball of radius $\sim N^{\frac12}$. Then \eqref{fio vnhtyrgwrugj8btm8-8=-2r93]xe=-34.or,0-cvr9yt0-} yields
$$N^{\frac{5}{2p}}\lesssim_{\epsilon,\nu}N^{\epsilon-\frac32+2},$$
which amounts to $p\ge 5$.
\end{remark}
Theorem \ref{tfek4}  implies  the following one, which we will prefer in our applications.

\begin{theorem}
\label{tfek1}For each $0<\nu\ll 1$, each $\nu-$transverse sets $S_1,\ldots,S_{10}\subset [0,1]^2$, each $g_j:S_j\to\C$, each ball $B_N\subset\R^5$ with radius $N\ge 1$ and each $\epsilon>0$ we have
$$\|(\prod_{j=1}^{10}E_{S_j}g_j)^{\frac1{10}}\|_{L^5(B_N)}\lesssim_{\epsilon,\nu} N^{\epsilon}(\prod_{j=1}^{10}\|g_j\|_{L^2(S_j)})^{\frac1{10}}.$$
\end{theorem}

To see that Theorem \ref{tfek4} implies  Theorem \ref{tfek1}, choose a positive Schwartz function $\eta$ on $\R^5$ such that
$$1_{B(0,1)}\le \eta,\;\;\text{ and }\;\;\supp\;\widehat{\eta}\subset B(0,\frac1{100}),$$
and let
\begin{equation}
\label{fek15}
\eta_{B_N}(x)=\eta(\frac{x-c(B_N)}{N}).
\end{equation}
Then, for $g_j$ as in Theorem \ref{tfek1},
$$\|(\prod_{j=1}^{10}E_{S_j}g_j)^{\frac1{10}}\|_{L^5(B_N)}\le \|(\prod_{j=1}^{10}((E_{S_j}g_j)\eta_{B_N}))^{\frac1{10}}\|_{L^5(B_N)}.$$
It suffices to note that the Fourier transform of $(E_{S_j}g_j)\eta_{B_N}$ is supported in $\A_{S_j,\frac1N}$ and that its $L^2$ norm is  $O(N^{-\frac32}\|g_j\|_2)$.
\bigskip

We record for future use the following consequence of Theorem \ref{tfek4}.
\begin{corollary}
\label{fc1}
Let $R_1,\ldots,R_{10}\subset [0,1]^2$ be $\nu$-transverse squares. Then for each $5\le p\le \infty$ and  $g_i:R_i\to\C$ we have
\begin{equation}
\label{we3}
\|(\prod_{i=1}^{10}\sum_{\atop{l(\Delta)=N^{-1/2}}} |E_{\Delta}g_i|^2)^{1/20}\|_{L^{p}(w_{B_N})}\lesssim_{\nu,\epsilon} N^{-\frac{15}{2p}+\epsilon}(\prod_{i=1}^{10}\sum_{\atop{l(\Delta)=N^{-1/2}}}\|E_{\Delta}g_i\|_{L^{\frac{2p}5}(w_{B_N})}^2)^{\frac1{20}}.
\end{equation}
\end{corollary}
\begin{proof}
 Using the function $\eta_{B_N}$ introduced earlier, together with  Theorem \ref{tfek1} and Plancherel's identity we get the following local inequality
$$\|(\prod_{i=1}^{10} |E_{R_i}g_i|)^{1/10}\|_{L^{5}(w_{B_N})}\lesssim_{\nu,\epsilon} N^{-\frac32+\epsilon}(\prod_{i=1}^{10}\sum_{\atop{l(\Delta)=N^{-1/2}}}\|E_{\Delta}g_i\|_{L^{2}(w_{B_N})}^2)^{\frac1{20}}.$$
A  randomization argument further leads to the inequality
$$\|(\prod_{i=1}^{10}\sum_{\atop{l(\Delta)=N^{-1/2}}} |E_{\Delta}g_i|^2)^{1/20}\|_{L^{5}(w_{B_N})}\lesssim_{\nu,\epsilon} N^{-\frac32+\epsilon}(\prod_{i=1}^{10}\sum_{\atop{l(\Delta)=N^{-1/2}}}\|E_{\Delta}g_i\|_{L^{2}(w_{B_N})}^2)^{\frac1{20}}.$$
It now suffices to interpolate this with the trivial inequality
$$\|(\prod_{i=1}^{10}\sum_{\atop{l(\Delta)=N^{-1/2}}} |E_{\Delta}g_i|^2)^{1/{20}}\|_{L^{\infty}(w_{B_N})}\le (\prod_{i=1}^{10}\sum_{\atop{l(\Delta)=N^{-1/2}}}\|E_{\Delta}g_i\|_{L^{\infty}(w_{B_N})}^2)^{\frac1{20}}.$$
We refer the reader to \cite{BD3} for how this type of interpolation is performed.
\end{proof}
\bigskip

The proof of Theorem \ref{tfek4} will be done in two stages. First, we reduce it to a statement about plates, a multilinear Kakeya-type inequality. The second part of this section is then devoted to proving this inequality.

Our arguments are immediate adaptations of those in \cite{BCT} and \cite{Gu}.

\subsection{Reduction to a multilinear Kakeya-type inequality}
The argument in this section is essentially the one  from Section 2 of \cite{BCT}.

We first prepare the ground for the proof of Theorem \ref{tfek4}. Consider a finitely overlapping cover of  $\M_{S_j}$ with $\frac1{N^{1/2}}\times \frac1{N^{1/2}}-$caps. Consider also the associated finitely overlapping cover $\P_{j,\frac1N}$ of $\A_{S_j,\frac1N}$ with thick caps having dimensions roughly $\frac1{N^{1/2}}\times \frac1{N^{1/2}}\times \frac1N\times \frac1N\times \frac1N$.
Each function $\widehat{f_j}$ as in Theorem \ref{tfek4} has a wave-packet decomposition of the form
\begin{equation}
\label{fek9}
\widehat{f_j}=\sum_{T\in\Q_j}c_T\phi_T.
\end{equation}
The coefficients $c_T$ are arbitrary complex numbers. The collection $\Q_j$ consists of rectangular parallelepipeds, which we will refer to as {\em plates},  with dimensions $N^{1/2}\times N^{1/2}\times N\times N\times N$. The two sides with length $N^{1/2}$ span a plane which is a translation  of  the plane spanned by the vectors $(1,0,2x,0,y)$ and $(0,1,0,2y,x)$, where $(x,y,x^2,y^2,xy)$ is the center of one of the caps covering $\M_{S_j}$. Thus all $T$ corresponding to a cap are translates of each other, and in fact they tile $\R^5$.

The function $\phi_T$ is a smooth approximation of $1_T$, whose Fourier transform is supported in some $\theta\in \P_{j,\frac1N}$. Moreover, $\|\phi_T\|_2=|T|^{1/2}=N^2.$ The functions $\phi_T,\;T\in\Q_j,$ are almost orthogonal, so that
\begin{equation}
\label{fek12}
\|f_j\|_2\sim N^2(\sum_{T\in \Q_j}|c_T|^2)^{1/2}.
\end{equation}

Let $r_T(\omega)$ be a subset of the Rademacher sequence, indexed by $T$. If we use random functions $$\widehat{f_j}=\sum_{T\in\Q_j}r_T(\omega)c_T\phi_T$$
in Theorem \ref{tfek4}, we get that
\begin{equation}
\label{fek8}
\|\prod_{j=1}^{10}(\sum_{T\in\Q_j}|c_T|^21_T)^{\frac1{10}}\|_{L^{\frac{5}2}(B_N)}\lesssim_{\epsilon,\nu} N^{-1+\epsilon}\prod_{j=1}^{10}(\sum_{T\in\Q_j}|c_T|^2)^{\frac1{10}}.
\end{equation}
We will prove that this multilinear Kakeya-type inequality is true in the next subsection. For now, we will assume \eqref{fek8} is true, and  we will  show that it implies Theorem \ref{tfek4}.

To avoid unnecessary technicalities, we will ignore the Schwartz tails of $\phi_T$ and will write
\begin{equation}
\label{fek10}
\phi_T(x)\sim 1_T(x)e(\xi_T\cdot x),
\end{equation}
where $\xi_T$ is the center of the corresponding cap. To make the argument formal, one needs to work with mollifications, as in section 2 from \cite{BCT}. The details are left to the interested reader.

Let $\MLR_N$ be the smallest constant such that
$$\|(\prod_{j=1}^{10}\widehat{f}_j)^{\frac1{10}}\|_{L^5(B_N)}\le \MLR_N(\prod_{j=1}^{10}\|f_j\|_{L^2(\A_{S_j,\frac1N})})^{\frac1{10}}$$
holds for each $f_j$ and each $B_N$ as in Theorem \ref{tfek4}. Our goal is to prove that
\begin{equation}
\label{fek14}
\MLR_N\lesssim_{\epsilon,\nu} N^{\epsilon-\frac32}, \text{ for each }\epsilon>0.
\end{equation}
This will follow by iterating the following inequality that we will prove next
\begin{equation}
\label{fek13}
\MLR_N\le \Sigma_{\nu,\epsilon} N^{-\frac34+\epsilon}\MLR_{N^{\frac12}}, \text{ for each }\epsilon>0,\;N\ge 1.
\end{equation}
Indeed, assume for the moment that \eqref{fek13} holds. Let $l$ be the largest integer so that $N^{\frac1{2^l}}\ge 2$. Note that $l\le \log_2\log_2 N$. Fix $\epsilon>0$. By applying \eqref{fek13} $l$ times we get that
$$\MLR_N\lesssim_\nu (\Sigma_{\nu,\epsilon})^{l}N^{-\frac34-\frac38-\ldots-\frac{3}{2^{l+1}}+l\epsilon},$$
as the value of $\MLR_N$ for $2\le N\le 4$ depends only on $\nu$. It is easy to see now that \eqref{fek14} holds.
\bigskip

For the remainder of the subsection, we prove $\eqref{fek13}$. Denote by $\Xi_{j,N}$ the centers of the caps with diameter $\frac1N$, forming a finitely overlapping cover of $\M_{S_j}$.
We observe that, using $$\widehat{f_j}(x)=\sum_{\xi\in\Xi_{j,N}}c_\xi \eta_{B_N}(x)e(\xi\cdot x)$$
with $\eta_{B_N}$ as in \eqref{fek15}, we get
\begin{equation}
\label{fek11}
\int_{B_N}\prod_{j=1}^{10}|\sum_{\xi\in\Xi_{j,N}}c_\xi e(\xi\cdot x)|^{\frac12}\lesssim(\MLR_N)^5N^{\frac{25}2}\prod_{j=1}^{10}(\sum_{\xi\in \Xi_{j,N}}|c_\xi|^2)^{\frac14}.
\end{equation}
We will next use this inequality with $N$ replaced by $N^{1/2}$. Namely, take $f_j$ as in \eqref{fek9}. For a fixed $B_N$, let $\B$ be a finitely overlapping cover of $B_N$ with balls $B$ with radius $N^{1/2}$. Using the heuristics in \eqref{fek10} we can write
$$\|(\prod_{j=1}^{10}\widehat{f}_j)^{\frac1{10}}\|_{L^5(B_N)}^5\sim \int_{B_N}\prod_{j=1}^{10}|\sum_{T\in\Q_j(B_N)}c_T e(\xi_T\cdot x)|^{\frac12}\le$$$$\le\sum_{B\in\B}\int_{B}\prod_{j=1}^{10}|\sum_{T\in\Q_j(B)}c_T e(\xi_T\cdot x)|^{\frac12},$$
where $\Q_j(B)$ are those plates in $\Q_j$ that intersect $B$. Note that there are $O(1)$ such plates parallel to a given plate (in other words, associated with a given cap). Thus, using \eqref{fek11} at the smaller scale $N^{1/2}$  we can write
$$\int_{B}\prod_{j=1}^{10}|\sum_{T\in\Q_j(B)}c_T e(\xi_T\cdot x)|^{\frac12}\lesssim (\MLR_{N^{\frac12}})^5N^{\frac{25}4}\prod_{j=1}^{10}(\sum_{T\in\Q_j(B)}|c_T|^2)^{\frac14}.$$
Summing up we get
$$\sum_{B\in\B}\int_{B}\prod_{j=1}^{10}|\sum_{T\in\Q_j(B)}c_T e(\xi_T\cdot x)|^{\frac12}\lesssim (\MLR_{N^{\frac12}})^5N^{\frac{25}4}\sum_{B\in\B}\prod_{j=1}^{10}(\sum_{T\in\Q_j(B)}|c_T|^2)^{\frac14}\sim$$
$$\sim (\MLR_{N^{\frac12}})^5N^{\frac{25}4}N^{-\frac52}\sum_{B\in\B}\int_B\prod_{j=1}^{10}(\sum_{T\in\Q_j(B)}|c_T|^21_T)^{\frac14}\sim$$
$$\sim(\MLR_{N^{\frac12}})^5N^{\frac{15}4}\int_{B_N}\prod_{j=1}^{10}(\sum_{T\in\Q_j(B_N)}|c_T|^21_T)^{\frac14}.$$
By invoking \eqref{fek8} we can dominate the above by
$$\lesssim_{\epsilon,\nu}N^{\frac{5}2+\epsilon}(\MLR_{N^{\frac12}})^5N^{\frac{15}4}\prod_{j=1}^{10}(\sum_{T\in\Q_j}|c_T|^2)^{\frac1{4}}.$$
Now, using \eqref{fek12}, the above is
$$\lesssim_{\epsilon,\nu}N^{\frac{5}2+\frac{15}4-10+\epsilon}(\MLR_{N^{\frac12}})^5\prod_{j=1}^{10}\|f_j\|_2^{\frac12}=
N^{-\frac{15}4+\epsilon}(\MLR_{N^{\frac12}})^5\prod_{j=1}^{10}\|f_j\|_2^{\frac12}.$$
This proves \eqref{fek13}.

\subsection{The proof of the multilinear Kakeya-type inequality}
\bigskip

The goal of this subsection is to prove \eqref{fek8}.

Let $\P$ be the collection of all 3-planes (three dimensional affine spaces) $P$ in $\R^5$ whose orthogonal complement (of the translated linear space) is a plane spanned by $(1,0,2x,0,y)$ and $(0,1,0,2y,x)$, with $(x,y)\in[0,1]^2$. We will say that $P$ is associated with $(x,y)$.
\begin{definition}
We will say that ten families $\P_i$ of 3-planes in $\P$ are $\nu-$transverse if for each $P_i\in\P_i$ associated with $(x_i,y_i)$, $1\le i\le 10$, the sets $S_i=\{(x_i,y_i)\}$ are $\nu-$transverse in the sense of Definition \ref{dfek1}.
\end{definition}

Suppose $P_{j,a}$ are elements of $\P$, for $1\le j\le 10$ and $1\le a\le N_j$. We allow repetitions within a family, so it may happen that $P_{j,a}=P_{j,a'}$ for some $a\not=a'$.

For $W\ge 1$,  we will denote by $T_{j,a,W}$ the characteristic function of the $W-$neighborhood of $P_{j,a}$. For simplicity, we will denote by $T_{j,a}$ the value of $T_{j,a,1}$. We will abuse earlier terminology and will also call $T_{j,a,W}$ plates. The fact that we allow these plates to be infinitely long in three orthogonal directions will allow for more elegant arguments, and will produce superficially stronger results.

We reduce \eqref{fek8} to the following multilinear Kakeya-type inequality.
\begin{theorem}
\label{tfek5}
Assume the ten families $\P_j=\{P_{j,a}:1\le a\le N_j\}$ are $\nu-$transverse. Let $B_S$ be any ball with radius $S\ge 1$ in $\R^5$. Then for any $\epsilon>0$ there exists $C_{\epsilon,\nu}>0$ such that for any $S\ge 1$  we have
$$\int_{B_S}\prod_{j=1}^{10}(\sum_{a=1}^{N_j}T_{j,a})^{\frac14}\le C_{\epsilon,\nu} S^\epsilon\prod_{j=1}^{10}N_j^{\frac14}.$$
\end{theorem}
Let us see why this theorem implies \eqref{fek8}. The first observation is that, under the hypothesis of Theorem \ref{tfek5}, the following superficially stronger inequality holds true for all $c_{j,a}\in [0,\infty).$
$$\int_{B_S}\prod_{j=1}^{10}(\sum_{a=1}^{N_j}c_{j,a}T_{j,a})^{\frac14}\le C_{\epsilon,\nu} S^\epsilon\prod_{j=1}^{10}(\sum_{a}c_{j,a})^{\frac14}.$$
This is because we have allowed repetitions among plates. Consider this inequality with  $S=N^{1/2}$, and then rescale $x\mapsto N^{-1/2}x$ to get
$$\int_{B_N}\prod_{j=1}^{10}(\sum_{a=1}^{N_j}c_{j,a}T_{j,a,N^{\frac12}})^{\frac14}\le N^{-\frac52}C_{\epsilon,\nu} N^{\frac\epsilon2}\prod_{j=1}^{10}(\sum_{a}c_{j,a})^{\frac14}.$$
Finally, note that this is slightly stronger than \eqref{fek8}, since the transversality is preserved under rescaling and since the plates here are infinite in three orthogonal directions.
\bigskip

For the remainder of this subsection we will focus on proving Theorem \ref{tfek5}. Our proof is an adaptation of the argument from \cite{Gu}. We start with the following consequence of Theorem \ref{tfek2}, covering the case when the plates within each family are translates of each other.
\begin{corollary}
\label{cfek3}
Assume the ten families $\P_j=\{P_{j,a}:1\le a\le N_j\}$ are $\nu-$transverse and that all 3-planes within the family $\P_j$ are associated with the same $(x_j,y_j)$. Then
$$\int_{\R^5}\prod_{j=1}^{10}(\sum_{a=1}^{N_j}T_{j,a,W})^{\frac14}\lesssim_\nu W^{5} \prod_{j=1}^{10}N_j^{\frac14}.$$
\end{corollary}
\begin{proof}
Let $V_j$ be the plane spanned by $(1,0,2x_j,0,y_j)$ and $(0,1,0,2y_j,x_j)$. Each $P_{j,a}$ has the equation $\pi_j(x)=v_{j,a}$ for some $v_{j,a}\in V_j$.

Apply Theorem \ref{tfek2} to $V_j$, using
$$g_j=(\sum_{a=1}^{N_j}1_{B(v_{j,a},W)})^{1/2}.$$
It suffices to note that $$g_j\circ\pi_j=(\sum_{a=1}^{N_j}T_{j,a,W})^{1/2},$$
and that $$\|g_j\|_{L^2(V_j)}\sim N_j^{\frac12}W.$$
\end{proof}
\bigskip

Given some $1>\delta>0$, we will now assume that for each $j$ there is $P_j\in\P$ so that the "angle`` $d_{\G(3,\R^5)}(P_j,P_{j,a})$ between $P_j$ and each $P_{j,a}\in\P_j$ is very small. By that we mean that for each ball $B\subset\R^5$ with radius $\le \delta^{-1}W$ ($W\ge 1$) and each $P_{j,a}\in\P_j$, there exists a translation of $P_j$, call it $\tilde{P}_{j,a,B}$,  so that
\begin{equation}
\label{fek6}
T_{j,a,W}(x)\le \tilde{T}_{j,a,B,2W}(x), \text{ for all }x\in B.
\end{equation}
Here $\tilde{T}_{j,a,B, W}$ denotes the $W-$neighborhood of $\tilde{P}_{j,a,B}$. The existence of such a small angle $\theta(\delta)$ is a consequence of elementary geometry.

Define $f_{j,W}:=\sum_{a=1}^{N_j}T_{j,a,W}$.

\begin{lemma}
\label{lfek4}
Let $\delta,W,\P_j$ be as above.
Assume that the 3-planes $P_j\in\P$ are $\nu-$transverse and that
\begin{equation}
\label{fek7}
d_{\G(3,\R^5)}(P_j,P_{j,a})\le \theta(\delta),
\end{equation}
for each $P_{j,a}\in\P_j$. Then for each ball $B_S\subset \R^5$ with radius $S\ge \delta^{-1}W$ we have for some $C_\nu$ depending only on $\nu$,
$$\int_{B_S}\prod_{j=1}^{10}f_{j,W}^{\frac14}\le C_\nu \delta^{5}\int_{B_S} \prod_{j=1}^{10}f_{j,\delta^{-1}W}^{\frac14}.$$
\end{lemma}
\begin{proof}
We consider a finitely overlapping cover of $B_S$ with balls $B$ of radius $\frac1{100}\delta^{-1}W$. For each such $B$, it suffices to prove that
$$\int_{B}\prod_{j=1}^{10}f_{j,W}^{\frac14}\lesssim_\nu \delta^{5}\int_{B} \prod_{j=1}^{10}f_{j,\delta^{-1}W}^{\frac14}.$$
Due to \eqref{fek6} we have
$$\int_{B}\prod_{j=1}^{10}f_{j,W}^{\frac14}\le \int_{B}\prod_{j=1}^{10}(\sum_a \tilde{T}_{j,a,B,2W})^{\frac14}.$$
Let $N_j(B)$ be the number of plates $\tilde{T}_{j,a,B,2W}$ that intersect $B$. Invoking Corollary \ref{cfek3} we get
$$\int_{B}\prod_{j=1}^{10}(\sum_a \tilde{T}_{j,a,B,2W})^{\frac14}\lesssim_\nu W^{5} \prod_{j=1}^{10}N_j(B)^{\frac14}.$$

Since the diameter of $B$ is  $\frac1{100}\delta^{-1}W$, and using again \eqref{fek6}, if $\tilde{T}_{j,a,B, 2W}$  intersects $B$, then ${T}_{j,a,\delta^{-1}W}$ is identically 1 on $B$. Thus we can write
$$W^{5} \prod_{j=1}^{10}N_j(B)^{\frac14}\le \delta^5\int_B\prod_{j=1}^{10}(\sum_{a}T_{j,a,\delta^{-1}W})^{\frac14},$$
and this concludes the argument.
\end{proof}

Iterating Lemma \ref{lfek4} we obtain the following result.
\begin{proposition}
\label{pfek2}
Assume the ten families $\P_j=\{P_{j,a}:1\le a\le N_j\}$ satisfy the requirements of Lemma \ref{lfek4} for some fixed $\delta$, $\nu$ and $C_\nu$. Then for  each ball $B_S$ with radius $S\ge 1$ in $\R^5$ we have
$$\int_{B_S}\prod_{j=1}^{10}(\sum_{a=1}^{N_j}T_{j,a})^{\frac14}\le \delta^{-5}S^{\frac{\log C_\nu}{\log \delta^{-1}}}\prod_{j=1}^{10}N_j^{\frac14}.$$
\end{proposition}
\begin{proof}
Since each $B_S$ can be covered by $\delta^{-5}$ balls with radius $\delta^{-M}\le S$ and $M$ a positive integer, it suffices to show that
$$\int_{B_S}\prod_{j=1}^{10}(\sum_{a=1}^{N_j}T_{j,a})^{\frac14}\le S^{\frac{\log C_\nu}{\log \delta^{-1}}}\prod_{j=1}^{10}N_j^{\frac14}$$
for  $B_S$ with side length $S=\delta^{-M}$.

Iterating Lemma \ref{lfek4} we get
$$\int_{B_S}\prod_{j=1}^{10}(\sum_{a=1}^{N_j}T_{j,a})^{\frac14}=\int_{B_S}\prod_{j=1}^{10}f_{j,1}^{\frac14}\le C_\nu\delta^{5}\int_{B_S}\prod_{j=1}^{10}f_{j,\delta^{-1}}^{\frac14}\le$$
$$\le\ldots\le (C_\nu\delta^{5})^{M}\int_{B_S}\prod_{j=1}^{10}f_{j,\delta^{-M}}^{\frac14}\le (C_\nu\delta^{5})^{M}S^5\prod_{j=1}^{10}N_j^{\frac14}=S^{\frac{\log C_\nu}{\log \delta^{-1}}}\prod_{j=1}^{10}N_j^{\frac14}.$$
\end{proof}
\bigskip

We are now ready to prove Theorem \ref{tfek5}.
\medskip

\begin{proof}[of Theorem \ref{tfek5}]
Given $\epsilon>0$, choose $\delta>0$ small enough so that $\frac{\log C_\nu}{\log \delta^{-1}}<\epsilon$. Using the compactness of $\G(3,\R^5)$, there is a number $N(\delta)$ so that we can split each family $\P_j$ into at most $N(\delta)$ subfamilies each of which satisfies \eqref{fek7}, for some $P_j$ that depends on the subfamily. We find that $\int_{B_S}\prod_{j=1}^{10}(\sum_{a=1}^{N_j}T_{j,a})^{\frac14}$ is dominated by the sum of $O(N(\delta)^{10})$ terms of the form
$\int_{B_S}\prod_{j=1}^{10}(\sum_{a=1}^{M_j}T_{j,a})^{\frac14}$
with $M_j\le N_j$. Moreover, each term can be bounded using Proposition \ref{pfek2} by
$$\delta^{-5}S^{\frac{\log C_\nu}{\log \delta^{-1}}}\prod_{j=1}^{10}M_j^{\frac14}\le \delta^{-5}S^{\epsilon}\prod_{j=1}^{10}N_j^{\frac14}.$$
It is now clear that $C_{\epsilon,\nu}=\delta^{-5}N(\delta)^{10}$ works, since $\delta$ depends only on $\epsilon$ and $\nu.$
\end{proof}

\bigskip

\section{The geometric argument}
\label{Comb}
\bigskip

A $K-$square will be a  square in $[0,1]^2$ with side length $\frac1K$. The collection of all dyadic $K-$squares will be denoted by $\Col_K$. For a $K-$square $R$ in $\Col_K$, we will denote by $2R$ the $2K-$square with the same center as $R$.

The main result in this section is the following theorem, whose relevance will be clear in the proof of Proposition \ref{fp2}.
\begin{theorem}
\label{tfek3}
For each $K\ge 1$ and $\epsilon>0$, there exists $\nu_{K}>0$ and there exists $\Lambda_\epsilon>0$ depending on $\epsilon$ but not on $K$,  so that  each subcollection of $\Col_K$ with at least $\Lambda_\epsilon K^{1+\epsilon}$ squares contains ten $\nu_{K}-$transverse squares.
\end{theorem}
This will follow from a sequence of auxiliary results. Given three squares $R_1,R_2,R_3\subset [0,1]^2$ define
$$\Reach_{lin}(R_1,R_2,R_3)=$$
$$\{(x,y)\in R_3: (x,y),(x_1,y_1),(x_2,y_2)\text{ are collinear for some }(x_1,y_1)\in R_1,(x_2,y_2)\in R_2\}\cup$$
$$\cup\{(x,y)\in R_3: y=y_1 \text{ for some }(x_1,y_1)\in R_1\}.$$
\bigskip

\begin{lemma}
\label{lfek1}
There exists $C_1$ such that given any integer $d\ge 2$ and any two $K-$squares $R_1,R_2$ that sit inside a $K^{\frac1{2^d}}-$square $R_3$, and which do not sit inside any  $K^{\frac1{2^{d-1}}}-$square, the set
$\Reach_{lin}(R_1,R_2,R_3)$
intersects at most $C_1K$ squares from the collection
$$\{2R:\;R\in \Col_K\}.$$

\end{lemma}
\begin{proof}
The proof  follows from very elementary geometry. The set $\Reach_{lin}(R_1,R_2,R_3)$ is the union of a horizontal strip of width $\frac1K$ and the part of a double cone with aperture $O(\frac{K^{\frac{1}{2^{d-1}}}}{K})$, both having diameter $O(K^{-\frac1{2^{d}}})$. The area of $\Reach_{lin}(R_1,R_2,R_3)$ is thus $O(\frac{1}K)$, and the conclusion follows.
\end{proof}
\bigskip

For each $K\ge 1$ let $\Col_{lin}(K)$  be the collection of all three-tuples $(R_1,R_2,R_3)$  with $R_1,R_2,R_3\in\Col_K$ and $2R_3\cap \Reach_{lin}(R_1,R_2,[0,1]^2)=\emptyset$. Define $\nu_{lin}(K)$ to be the infimum of
$$|(y-y_1)[(x-x_1)(y_2-y_1)-(x_2-x_1)(y-y_1)]\big|$$
taken over all points such that
$$(x_1,y_1)\in R_1,\;(x_2,y_2)\in R_2,\;(x,y)\in R_3,$$
with $(R_1,R_2,R_3)\in \Col_{lin}(K)$. By invoking a compactness argument, it is easy to see that
$\nu_{lin}(K)>0.$
\bigskip

For each $K\ge 1$ let $\Col_{quad}(K)$ be the collection of all five-tuples $(R_1,\dots,R_5)$ in $\Col_K$ so that given any $i,j,k\in\{1,2,3,4,5\}$, there is a permutation $\pi:\{i,j,k\}\to\{i,j,k\}$ such that $(R_{\pi(i)}, R_{\pi(j)},R_{\pi(k)})\in\Col_{lin}(K)$.
Recall the definition of $Q_{u,v}$ from Section \ref{Trans}. Let
$$\nu_{quad}(K)=\inf_{u,v,w}\max_j(|Q_{u,v}(x_j,y_j)|+|Q_{u,w}(x_j,y_j)|)$$
where the infimum is taken over all orthonormal triples $u,v,w$ in $\R^5$, and the maximum is taken over all points $(x_j,y_j)\in R_j$
with $(R_1,\ldots,R_5)\in \Col_{quad}(K)$. As observed earlier, the intersection of the zero sets of $Q_{u,v}$ and $Q_{u,w}$ is either a line, or a finite set with at most four points.
By invoking a compactness argument and the continuity of $Q_{u,v}(x,y)$ in $u,v,x,y$, it is easy to see that
$\nu_{quad}(K)>0.$

\bigskip

Define now
$$\nu_K=\min\{\nu_{lin}(K),\nu_{quad}(K)\}.$$Let $C_2$ be a large enough constant, independent of $K$ ($10^{10}$ probably works).
\begin{lemma}
\label{lfek2}
 If there is a $K^{1/4}-$square $R$ containing at least $C_2C_1K$ squares from $\Col_K$, then among these squares we can find ten which are $\nu_K-$transverse.
\end{lemma}
\begin{proof}
The selection is inductive. Start with any square $R_1\in \Col_K$. Assume we have selected $m-1\le 9$ squares $R_1,\ldots, R_{m-1}$ which are $\nu_K-$transverse. We select the next square $R_{m}\in\Col_K$ subject to the following restrictions

(i) $R_m\subset R$

(ii) For each $1\le i\le m-1$, $R_m$ and $R_i$ do not sit inside a  $K^{1/2}-$square

(iii) $2R_m\cap\Reach_{lin}(R_i,R_j,R)=\emptyset$, for each $1\le i\not= j\le m-1$

\medskip

Note that (ii) forbids the selection of  $O(K)$ squares.

Now, Lemma \eqref{lfek1} with $d=2$ shows that among the squares satisfying  (i), the requirements  (ii) and (iii) are  satisfied for all but  $O(C_1K)$ squares. The conclusion follows if $C_2$ is large enough.
\end{proof}

An immediate consequence is the proof of Theorem \ref{tfek3} when $\epsilon=\frac12$.
\begin{corollary}
\label{cfek1}
Any subset of $\Col_K$ with at least $C_2C_1K^{1+\frac12}$ squares contains ten which are $\nu_K-$transverse.
\end{corollary}
\begin{proof}
The hypothesis implies that there is a $K^{1/4}-$square that contains at least $C_2C_1K$ squares from $\Col_K$, so Lemma \ref{lfek2} applies.
\end{proof}
We repeat the above reasoning  as follows.
\begin{lemma}
\label{lfek3}
Let $d\ge 2$.
If there is a $K^{1/2^d}-$square $R$ containing at least $100^{d-2}C_2C_1K$ squares from $\Col_K$, then among these squares we can find ten which are $\nu_K-$transverse.
\end{lemma}
\begin{proof}
The proof is by induction on $d$. We have already seen the case $d=2$. Assume we have verified the lemma for some $d-1\ge 2$. Consider a collection satisfying the hypothesis.

We distinguish two cases. First, if there is a smaller $K^{1/2^{d-1}}-$square $R'$ containing at least $100^{d-3}C_2C_1K$ squares from $\Col_K$, the conclusion follows from our induction hypothesis.

We can thus assume that each $K^{1/2^{d-1}}-$square contains at most $100^{d-3}C_2C_1K$ squares from $\Col_K$.
The selection of the ten squares is inductive, essentially identical to the one from Lemma \ref{lfek2}. Start with any square $R_1$. Assume we have selected $m-1\le 9$ squares $R_1,\ldots, R_{m-1}$ which are $\nu_K-$transverse. We select the next square $R_{m}\in\Col_K$ subject to the following restrictions

(i) $R_m\subset R$

(ii) for each $1\le i\le m-1$,  $R_m$ and $R_i$ do not sit inside a $K^{1/2^{d-1}}-$square

(iii) $2R_m\cap\Reach_{lin}(R_i,R_j,R)=\emptyset$, for each $1\le i\not= j\le m-1$.

Note that due to our assumption, (ii) forbids the selection of  $O(100^{d-3}C_2C_1K)$ squares.

Now, Lemma \ref{lfek1}  shows that among the squares satisfying  (i), the requirements  (ii) and (iii) are  satisfied for all but  $O(C_1K)$ squares. The conclusion now follows since the original collection contains sufficiently many square, in particular
$$O(C_1K)+O(100^{d-3}C_2C_1K)<100^{d-2}C_2C_1K.$$
\end{proof}
\begin{corollary}
\label{cfek2}
Any subset of $\Col_K$ with at least $100^{d-1}C_2C_1K^{1+\frac1{2^d}}$ squares contains ten which are $\nu_K-$transverse.
\end{corollary}
\begin{proof}
The hypothesis implies that there is a $K^{\frac1{2^{d+1}}}-$square that contains at least $100^{d-1}C_2C_1K^{1}$ squares from $\Col_K$, so Lemma \ref{lfek3} applies.

\end{proof}

The proof of Theorem \ref{tfek3} is now immediate. For each $\epsilon>0$, let $d$ be the largest integer such that $\epsilon\le \frac1{2^d}$. Define now $\Lambda_\epsilon=100^{d-1}C_2C_1$.

\bigskip

\section{Linear versus $10-$linear decoupling}

In the remaining part of the paper, we will follow the approach from \cite{BD5}.

First, we recall the following ``trivial" decoupling from \cite{BD5}, that we will use  to bound the non transverse contribution in the Bourgain--Guth decomposition. For completeness, we reproduce the proof from \cite{BD5}.
\begin{lemma}
\label{fl1}
Let $R_1,\ldots,R_M$ be pairwise disjoint squares in $[0,1]^2$ with side length $K^{-1}$. Then for each $2\le p\le \infty$
$$ \|\sum_iE_{R_i}g\|_{L^p(w_{B_K})}\lesssim_p M^{1-\frac2p}(\sum_i\|E_{R_i}g\|_{L^p(w_{B_K})}^p)^{1/p}.$$
\end{lemma}
\begin{proof}
The key observation is the fact that if $f_1,\ldots,f_M:\R^5\to\C$ are such that $\widehat{f_i}$ is supported on a ball $B_i$ and the dilated balls $(2B_i)_{i=1}^M$ are pairwise disjoint, then
\begin{equation}
\label{fe36}
\|f_1+\ldots+f_M\|_{L^p(\R^5)}\lesssim_p  M^{1-\frac2p}(\sum_i\|f_i\|_{L^p(\R^5)}^p)^{\frac1p}.
\end{equation}
In fact more is true. If $T_i$ is a smooth Fourier multiplier adapted to $2B_i$ and equal to 1 on $B_i$, then the inequality
$$\|T_1(f_1)+\ldots+T_M(f_M)\|_{L^p(\R^5)}\lesssim_p  M^{1-\frac2p}(\sum_i\|f_i\|_{L^p(\R^5)}^p)^{\frac1p}$$
for arbitrary $f_i\in L^p(\R^5)$
follows by interpolating the immediate $L^2$ and $L^\infty$ estimates. Inequality \eqref{fe36} is the best one can say in general, if no further assumption is made on the Fourier supports of $f_i$. Indeed, if $\widehat{f_i}=1_{B_i}$ with $B_i$ equidistant balls of radius one with collinear centers, then the reverse inequality will hold.

Let now $\eta_{B_K}$ be as in \eqref{fek15}. It suffices to note that the Fourier supports of the  functions $f_i=\eta_{B_K}E_{R_i}g$ have bounded overlap.
 \end{proof}
\bigskip

For $2\le p<\infty$ and $N\ge 1$,  recall that $D(N,p)$ is the smallest constant such that the decoupling
$$\|E_{[0,1]^2}g\|_{L^p(w_{B_N})}\le D(N,p)(\sum_{l(\Delta)=N^{-1/2}}\|E_{\Delta}g\|_{L^p(w_{B_N})}^p)^{1/p}$$
holds true for all $g$ and all balls $B_N$ or radius $N$.

We now introduce a $10-$linear version of $D(N,p) $. Given also $\nu\ll 1$, let $D_{multi}(N,p,\nu)$ be the smallest constant such that the inequality
$$\||\prod_{i=1}^{10}E_{R_i}g_j|^{\frac1{10}}\|_{L^p(w_{B_N})}\le D_{multi}(N,p,\nu)(\prod_{i=1}^{10}\sum_{l(\Delta)=N^{-1/2}}\|E_{\Delta}g_i\|_{L^p(w_{B_N})}^p)^{\frac1{10p}}$$
holds true for all $\nu$-transverse squares (see Definition \ref{dfek1}) $R_1,\ldots,R_{10}\subset [0,1]^{2}$ with equal, but otherwise arbitrary side lengths, all $g_i:R_i\to\C$  and all balls $B_N\subset\R^5$ with radius $N$.
\bigskip

H\"older's inequality shows that $D_{multi}(N,p,\nu)\le D(N,p)$. The rest of the section will be devoted to proving that the reverse inequality is also essentially true. This will  follow from a  variant of the Bourgain--Guth induction on scales in \cite{BG}. More precisely, we prove the following result. Recall the definition of $\nu_{K}$ from Theorem \ref{tfek3}.

\begin{theorem}
\label{ft2}
For each $K\ge 2$, $\epsilon>0$ and $p\ge 2$ there exists $\Lambda_{K,p,\epsilon}>0$ and $\beta(K,p,\epsilon)>0$  with
$$\lim_{K\to \infty}\beta(K,p,\epsilon)=0,\;\;\text{ for each }p,\epsilon,$$
such that for each $N\ge K$

$$D(N,p)\le N^{\beta(K,p,\epsilon)+(1+\epsilon)(\frac12-\frac1p)}+$$
 \begin{equation}
\label{fe31}+\Lambda_{K,p,\epsilon}\log_K N\max_{1\le M\le N}\left[(\frac{M}{N})^{(1+\epsilon)(\frac1p-\frac12)}D_{multi}(M,p,\nu_{K})\right].
\end{equation}
\end{theorem}
Recall that due to \eqref{fe30} we have
$D(N,p)\gtrsim N^{\frac12-\frac1p}$ for $p\ge 2$.
We conclude that the term $N^{\beta(K,p,\epsilon)+(1+\epsilon)(\frac12-\frac1p)}$ in \eqref{fe31} is rather harmless.
\bigskip

The key step in proving Theorem \ref{ft2}  is the following inequality.

\begin{proposition}
\label{fp2}
For $2\le p<\infty$ and each $\epsilon>0$ there is a constant $C_{p,\epsilon}$  so that for each $g$ and $N\ge K\ge 1$ we have
$$\|E_{[0,1]^2}g\|_{L^p(w_{B_N})}^p\le $$$$ C_{p,\epsilon}K^{(p-2)(1+\epsilon)}\sum_{R\in\Col_K}\|E_{R}g\|_{L^p(w_{B_N})}^p+C_{p,\epsilon}K^{100p}D_{multi}(N,p,\nu_{K})^p\sum_{\Delta\in
\Col_{N^{\frac12}}}\|E_{\Delta}g\|_{L^p(w_{B_N})}^p.$$
\end{proposition}
\bigskip

The exponent $100p$ in $K^{100p}$ is not important and could easily be improved, but the exponent $p-2$ in $K^{p-2}$ is sharp and will play a critical role in the rest of the argument.
\bigskip

\begin{proof}
Following the standard formalism from \cite{BG}, we may assume that $|E_{R}g(x)|$ is essentially constant on each ball $B_K$ of radius $K$, and will we denote by $|E_{R}g(B_K)|$ this value. Write
$$E_{[0,1]^2}g(B_K)=\sum_{R\in\Col_K}E_{R}g(B_K).$$
Fix $B_K$. Let $R^*\in \Col_K$ be a square which maximizes the value of $|E_{R}g(B_K)|$. Let $\Col_{B_K}^{*}$ be those squares  $R\in \Col_K$ such that
$$|E_{R}g(B_K)|\ge K^{-2}|E_{R^*}g(B_K)|.$$

We distinguish two cases.

First, if $\Col_{B_K}^{*}$ contains at least $\Lambda_\epsilon K^{1+\epsilon}$ squares, then invoking Theorem \ref{tfek3} we infer that $\Col_{B_K}^{*}$ contains ten $\nu_{K}-$transverse squares $R_1,\ldots,R_{10}$. In this case we can write
$$|E_{[0,1]^2}g(B_K)|\le K^4(\prod_{i=1}^{10}|E_{R_i}g(B_K)|)^{\frac1{10}}.$$

Otherwise, if $\Col_{B_K}^{*}$ contains at most $\Lambda_\epsilon K^{1+\epsilon}$ squares, we can write using the triangle inequality
$$|E_{[0,1]^2}g(B_K)|\le 2|E_{R^*}g(B_K)|+|\sum_{R\in\Col_{B_K}^{*}}E_{R}g(B_K)|.$$
Next, invoking Lemma \ref{fl1} we get
$$\|E_{[0,1]^2}g\|_{L^p(w_{B_K})}\lesssim_p\|E_{R^*}g\|_{L^p(w_{B_K})}+ (\Lambda_\epsilon K^{1+\epsilon})^{1-\frac2{p}}(\sum_{R\in\Col_{B_K}^{*}}\|E_{R}g\|_{L^p(w_{B_K})}^p)^{1/p}\le $$
$$\lesssim_{p,\epsilon}  K^{(1+\epsilon)(1-\frac2{p})}(\sum_{R\in\Col_{K}}\|E_{R}g\|_{L^p(w_{B_K})}^p)^{1/p}.$$
To summarize, in either case we can write
$$\|E_{[0,1]^2}g\|_{L^p(w_{B_K})}\lesssim_{p,\epsilon} $$$$ K^4\max_{R_1,\ldots,R_{10}:\;\nu_{K}-\text{transverse}}\|(\prod_{i=1}^{10}|E_{R_i}g|)^{1/10}\|_{L^p(w_{B_K})}+ K^{(1+\epsilon)(1-\frac2{p})}(\sum_{R\in\Col_K}\|E_{R}g\|_{L^p(w_{B_K})}^p)^{1/p}\le$$
$$K^4(\sum_{R_1,\ldots,R_{10}:\;\nu_{K}-\text{transverse}}\|(\prod_{i=1}^{10}|E_{R_i}g|)^{1/10}\|_{L^p(w_{B_K})}^p)^{1/p}+ K^{(1+\epsilon)(1-\frac2{p})}(\sum_{R\in\Col_K}\|E_{R}g\|_{L^p(w_{B_K})}^p)^{1/p}.$$
Raising to the power $p$  and summing over $B_K$ in a  finitely overlapping cover of $B_N$, leads to the desired conclusion.
\end{proof}
\bigskip

\bigskip

Using a form of parabolic rescaling,  the result in Proposition \ref{fp2} leads to the following general result.
\begin{proposition}
\label{fp3}
Let $R\subset[0,1]^2$ be a square with side length $\delta$. Then for each $\epsilon>0$ and each  $2\le p<\infty$, $g:R\to\C$, $K\ge 1$ and $N>\delta^{-2}$ we have
$$\|E_{R}g\|_{L^p(w_{B_N})}^p\le $$$$ C_{p,\epsilon}K^{(1+\epsilon)(p-2)}\sum_{R'\subset R\atop{R'\in\Col_{\frac{K}\delta}}}\|E_{R'}g\|_{L^p(w_{B_N})}^p+C_{p,\epsilon}K^{100p}D_{multi}(N\delta^{2},p,\nu_{K})^p\sum_{\Delta\subset R\atop{\Delta\in\Col_{N^{\frac12}}}}\|E_{\Delta}g\|_{L^p(w_{B_N})}^p,$$
where $C_{p,\epsilon}$ is the constant from Proposition \ref{fp2}.
\end{proposition}
\begin{proof}
Assume $R=[a,a+\delta]\times [b,b+\delta]$. The affine change of variables
$$(t,s)\in R\mapsto(t',s')=\eta(t,s)=(\frac{t-a}{\delta},\frac{s-b}{\delta})\in[0,1]^2$$ shows that
$$|E_Rg(x)|=\delta^2|E_{[0,1]^2}g^{a,b}(\bar{x})|,$$
$$|E_{R'}g(x)|=\delta^2|E_{R''}g^{a,b}(\bar{x})|,$$
where
$R''=\eta(R')$ is a square with side length $\frac1K$,
$$g^{a,b}(t',s')=g(\delta t'+a,\delta s'+b),$$
and the relation between $x=(x_1,\ldots,x_5)$ and $\bar{x}=(\bar{x}_1,\ldots,\bar{x}_5)$ is given by
$$\bar{x}_1=\delta(x_1+2ax_3+bx_5),$$
$$\bar{x}_2=\delta(x_2+2bx_4+ax_5),$$
$$\bar{x}_3=\delta^2x_3,\;\;\bar{x}_4=\delta^2x_4,\;\;\bar{x}_5=\delta^2x_5. $$
Note that $\bar{x}$ is the image of $x$ under a shear transformation. Call $C_N$ the image of the ball $B_N$ in $\R^5$ under this transformation. Cover $C_N$ with a family $\F$ of balls $B_{\delta^2N}$ with $O(1)$ overlap. Write
$$\|E_{R}g\|_{L^p(w_{B_N})}=\delta^{2-\frac8p}\|E_{[0,1]^2}g^{a,b}\|_{L^p(w_{C_N})}$$
for an appropriate weight $w_{C_N}$. The right hand side is bounded by
$$\delta^{2-\frac8p}(\sum_{B_{\delta^2N}\in\F}\|E_{[0,1]^2}g^{a,b}\|_{L^p(w_{B_{\delta^2N}})}^p)^{1/p}.$$
Apply Proposition \ref{fp2} to each of the terms $\|E_{[0,1]^2}g^{a,b}\|_{L^p(w_{B_{\delta^2N}})}$ and then rescale back.

\end{proof}

\bigskip

We are now in position to prove Theorem \ref{ft2}. By iterating Proposition \ref{fp3} $n$ times we get
$$\|E_{[0,1]^2}g\|_{L^p(w_{B_N})}^p\le (C_{p,\epsilon}K^{(1+\epsilon)(p-2)})^n\sum_{R\in\Col_{K^n}}\|E_{R}g\|_{L^p(w_{B_N})}^p+$$$$+C_{p,\epsilon}K^{100p}\sum_{\Delta\in\Col_{N^{1/2}}}
\|E_{\Delta}g\|_{L^p(w_{B_N})}^p\sum_{j=0}^{n-1}(C_{p,\epsilon}K^{(1+\epsilon)(p-2)})^{j}D_{multi}(NK^{-2j},p,\nu_{K})^p.$$
Applying this with $n$ such that $K^n=N^{\frac12}$ we get
$$\|E_{[0,1]^2}g\|_{L^p(w_{B_N})}\le $$$$ N^{\frac1{p}\log_{K}C_{p,\epsilon}}N^{^{(1+\epsilon)(\frac12-\frac1p)}}(\sum_{\Delta\in\Col_{N^{1/2}}}
\|E_{\Delta}g\|_{L^p(w_{B_N})}^p)^{1/p}+
$$$$C_{p,\epsilon}K^{100}\sum_{j=0}^{n-1}(\frac{NK^{-2j}}{N})^{(1+\epsilon)(\frac1p-\frac12)}
D_{multi}(NK^{-2j},p,\nu_{K})(\sum_{\Delta\in\Col_{N^{1/2}}}\|E_{\Delta}g\|_{L^p(w_{B_N})}^p)^{1/p}.
$$

The proof of Theorem \ref{ft2} is now complete, by taking $$\beta_{K,p,\epsilon}=\frac1{p}\log_{K}C_{p,\epsilon}$$
and
$$\Lambda_{K,p,\epsilon}=\frac12C_{p,\epsilon}K^{100}.$$

\bigskip

\section{The proof of Theorem \ref{tfek6}}
In this section we finish the proof of Theorem \ref{tfek6}, by showing that
$$D(N,8)\lesssim_\epsilon N^{\frac3{8}+\epsilon}.$$
For $p\ge 5$ define $\kappa_p$ such that
$$\frac5{2p}=\frac{1-\kappa_p}{2}+\frac{\kappa_p}{p},$$
in other words, $$\kappa_p=\frac{p-5}{p-2}.$$

\begin{proposition}
Let $R_1,\dots,R_{10}$ be $\nu$-transverse squares in $[0,1]^2$ with arbitrary side lengths.
We have that for each radius $R\ge N$, $p\ge 5$ and $g_i:R_i\to \C$
$$\|(\prod_{i=1}^{10}\sum_{\atop{l(\tau)=N^{-1/4}}}|E_{\tau}g_i|^2)^{\frac1{20}}\|_{L^{p}(w_{B_R})}\lesssim_{\nu,p,\epsilon}$$
$$
\lesssim_{\nu,p,\epsilon}N^{\epsilon}\|(\prod_{i=1}^{10}\sum_{\atop{l(\Delta)=N^{-1/2}}}|E_{\Delta}g_i|^2)^{\frac1{20}}\|_{L^{p}(w_{B_R})}^{1-\kappa_p}
(\prod_{i=1}^{10}\sum_{\atop{l(\tau)=N^{-1/4}}}\|E_{\tau}g_i\|_{L^{p}(w_{B_R})}^2)^{\frac{\kappa_p}{20}}.
$$
\end{proposition}
\begin{proof}
Let  $B$ be an arbitrary ball of radius $N^{1/2}$. We start by recalling that \eqref{we3} on $B$ gives
\begin{equation}
\label{fe20}
\|(\prod_{i=1}^{10}\sum_{\atop{l(\tau)=N^{-1/4}}} |E_{\tau}g_i|^2)^{1/20}\|_{L^{p}(w_{B})}\lesssim_{\nu,\epsilon,p} N^{-\frac{15}{4p}+\epsilon}(\prod_{i=1}^{10}\sum_{\atop{l(\tau)=N^{-1/4}}}\|E_{\tau}g_i\|_{L^{2p/5}(w_{B})}^2)^{\frac1{20}}.
\end{equation}
Write using H\"older's inequality
\begin{equation}
\label{we7}
(\sum_{\atop{l(\tau)=N^{-1/4}}}\|E_{\tau}g_i\|_{L^{p/2}(w_{B})}^2)^{\frac1{2}}\le (\sum_{\atop{l(\tau)=N^{-1/4}}}\|E_{\tau}g_i\|_{L^{2}(w_{B})}^2)^{\frac{1-\kappa_p}{2}}(\sum_{\atop{l(\tau)=N^{-1/4}}}\|E_{\tau}g_i\|_{L^{p}(w_{B})}^2)^{\frac{\kappa_p}{2}}.
\end{equation}
The next key element in our argument is the almost orthogonality specific to $L^2$, which will allow us to pass from scale $N^{-1/4}$ to scale $N^{-1/2}$. Indeed, since $(E_\Delta g_i)w_{B}$ are almost orthogonal for $l(\Delta)=N^{-1/2}$, we have
$$(\sum_{\atop{l(\tau)=N^{-1/4}}}\|E_{\tau}g_i\|_{L^{2}(w_{B})}^2)^{1/2}\lesssim (\sum_{\atop{l(\Delta)=N^{-1/2}}}\|E_{\Delta}g_i\|_{L^{2}(w_{B})}^2)^{1/2}.$$
We can now rely on the fact that $|E_{\Delta}g_i|$ is essentially constant on balls $B'$ of radius $N^{1/2}$ to argue that
$$(\sum_{\atop{l(\Delta)=N^{-1/2}}}\|E_{\Delta}g_i\|_{L^{2}(B')}^2)^{\frac1{2}}\sim |B'|^{1/2}(\sum_{\atop{l(\Delta)=N^{-1/2}}}|E_{\Delta}g_i(x)|^2)^{\frac1{2}}\text{ for }x\in{B'}$$
and thus
\begin{equation}
\label{we8}
(\prod_{i=1}^{10}\sum_{\atop{l(\Delta)=N^{-1/2}}}\|E_{\Delta}g_i\|_{L^{2}(w_{B})}^2)^{\frac1{20}}\lesssim |B|^{\frac12-\frac1p}\|(\prod_{i=1}^{10}\sum_{\atop{l(\Delta)=N^{-1/2}}}|E_{\Delta}g_i|^2)^{\frac1{20}}\|_{L^{p}(w_{B})}.
\end{equation}
Combining \eqref{fe20}, \eqref{we7} and \eqref{we8} we get
$$
\|(\prod_{i=1}^{10}\sum_{\atop{l(\tau')=N^{-1/4}}}|E_{\tau}g_i|^2)^{\frac1{20}}\|_{L^{p}(w_{B})}\lesssim_{\nu,p,\epsilon}$$$$\lesssim_{\nu,p,\epsilon}N^{\epsilon} \|(\prod_{i=1}^{10}\sum_{\atop{l(\Delta)=N^{-1/2}}}|E_{\Delta}g_i|^2)^{\frac1{20}}\|_{L^{p}(w_{B})}^{1-\kappa_p}
(\prod_{i=1}^{10}\sum_{\atop{l(\tau)=N^{-1/4}}}\|E_{\tau}g_i\|_{L^{p}(w_{B})}^2)^{\frac{\kappa_p}{20}}.
$$
Summing this up over a finitely overlapping family of balls $B\subset B_R$ of radius $N^{1/2}$, we get the desired inequality.
\end{proof}
\medskip

We will iterate the result of the above proposition in the following form, a consequence of the Cauchy--Schwartz inequality
$$\|(\prod_{i=1}^{10}\sum_{\atop{l(\tau)=N^{-1/4}}}|E_{\tau}g_i|^2)^{\frac1{20}}\|_{L^{p}(w_{B_R})}\le$$
\begin{equation}
\label{fe21}
\le C_{p,\nu,\epsilon}N^{\frac{\kappa_p}{2}(\frac12-\frac1p)+\epsilon}\|(\prod_{i=1}^{10}\sum_{\atop{l(\Delta)=N^{-1/2}}}
|E_{\Delta}g_i|^2)^{\frac1{20}}\|_{L^{p}(w_{B_R})}^{1-\kappa_p}(\prod_{i=1}^{10}\sum_{\atop{l(\tau)=N^{-1/4}}}
\|E_{\tau}g_i\|_{L^{p}(w_{B_R})}^p)^{\frac{\kappa_p}{10p}}.
\end{equation}

\bigskip

We will also need the following immediate consequence of the Cauchy--Schwartz inequality. While the exponent $2^{-s}$ in $N^{2^{-s}}$ can be improved if transversality is imposed, the following trivial estimate will suffice for our purposes.
\begin{lemma}
\label{wlem0081}Consider ten  rectangles $R_1,\ldots,R_{10}\subset [0,1]^2$ with arbitrary side lengths. Assume $g_i$ is supported on $R_i$.
Then for $1\le p\le\infty$ and $s\ge 2$
$$\|(\prod_{i=1}^{10}|E_{R_i}g_i|)^{1/10}\|_{L^{p}({w_{B_N}})}\le N^{2^{-s}}\|(\prod_{i=1}^{10}\sum_{\atop{l(\tau_s)=N^{-2^{-s}}}}|E_{\tau_s}g_i|^2)^{\frac1{20}}\|_{L^{p}(w_{B_N})}.$$
\end{lemma}

\bigskip

Using  parabolic rescaling as in the proof of Theorem \ref{fp3}, we get that for each square $R\subset[0,1]^2$ with side length $N^{-\rho}$,  $\rho\le \frac12$
\begin{equation}
\label{fe22}
\|E_Rg\|_{L^p(w_{B_N})}\le D(N^{1-2\rho},p)(\sum_{\Delta\subset R\atop{l(\Delta)=N^{-1/2}}}\|E_\Delta g\|_{L^p(w_{B_N})}^p)^{1/p}.
\end{equation}
\bigskip

Fix $\epsilon>0$, $K\ge 2$, to be chosen later. Recall the definition of $\nu_{K}$ from Theorem \ref{ft2}. For simplicity, we will denote the constant $C_{p,\nu_{K},\epsilon}$ from \eqref{fe21} with $C_{p,K,\epsilon}$.

Let $R_1,\ldots, R_{10}\subset [0,1]^2$ be $\nu_{K}$-transverse  rectangles with arbitrary side lengths and assume $g_i$ is supported on $R_i$. Start with Lemma \ref{wlem0081}, continue with iterating \eqref{fe21} $s-1$ times, and invoke \eqref{fe22} at each step to write
$$\|(\prod_{i=1}^{10}|E_{R_i}g_i|)^{1/2}\|_{L^{p}({B_N})}\le N^{2^{-s}}(C_{p,K,\epsilon}N^\epsilon)^{s-1}(\prod_{i=1}^{10}\sum_{\atop{l(\Delta)=N^{-1/2}}}
\|E_{\Delta}g_i\|_{L^{p}(w_{B_N})}^p)^{\frac{1}{10p}}\times$$
$$\times N^{\frac{\kappa_p}{2}(\frac12-\frac1p)(1-\kappa_p)^{s-2}}\cdot\ldots\cdot N^{\frac{\kappa_p}{2^{s-2}}(\frac12-\frac1p)(1-\kappa_p)}N^{\frac{\kappa_p}{2^{s-1}}(\frac12-\frac1p)}\|(\prod_{i=1}^{10}\sum_{\atop{l(\tau)=N^{-1/2}}}|E_{\Delta}g_i|^2)^{\frac1{20}}\|_{L^{p}(w_{B_N})}^{(1-\kappa_p)^{s-1}}\times$$
\begin{equation}
\label{fe34}
\times D(N^{1-2^{-s+1}},p)^{\kappa_p}D(N^{1-2^{-s+2}},p)^{\kappa_p(1-\kappa_p)}\cdot\ldots\cdot D(N^{1/2},p)^{\kappa_p(1-\kappa_p)^{s-2}}.
\end{equation}

Note that the inequality
$$\|(\sum_{\atop{l(\Delta)=N^{-1/2}}}|E_{\Delta}g_i|^2)^{\frac1{2}}\|_{L^{p}(w_{B_N})}\le N^{\frac12-\frac1p}(\sum_{\atop{l(\Delta)=N^{-1/2}}}\|E_{\Delta}g_i\|_{L^{p}(w_{B_N})}^p)^{1/p}$$
is an immediate  consequence of Minkowski's and H\"older's inequalities. Using this, \eqref{fe34} has the following consequence
$$D_{multi}(N,p,\nu_{K})\le (C_{p,K,\epsilon}N^\epsilon)^{s-1} N^{2^{-s}}N^{\kappa_p 2^{-s}(1-\frac2p)\frac{1-(2(1-\kappa_p))^{s-1}}{2\kappa_p-1}}\times$$
\begin{equation}
\label{fe23}
\times D(N^{1-2^{-s+1}},p)^{\kappa_p}D(N^{1-2^{-s+2}},p)^{\kappa_p(1-\kappa_p)}\cdot\ldots\cdot D(N^{1/2},p)^{\kappa_p(1-\kappa_p)^{s-2}}    N^{O_p((1-\kappa_p)^s)}.
\end{equation}
\bigskip

Let $\gamma_p$ be the unique positive number such that
$$\lim_{N\to\infty}\frac{D(N,p)}{N^{\gamma_p+\delta}}=0,\;\text{for each }\delta>0$$
and
\begin{equation}
\label{fe27}
\limsup_{N\to\infty}\frac{D(N,p)}{N^{\gamma_p-\delta}}=\infty,\;\text{for each }\delta>0.
\end{equation}
The existence of such $\gamma_p$ is guaranteed by \eqref{fe30} and \eqref{fe3008}.
Recall that our goal is to prove that $\gamma_8=\frac3{8}.$
By using the fact that $D(N,p)\lesssim_\delta N^{\gamma_p+\delta}$ in \eqref{fe23}, it follows that for each $\delta>0$ and $s\ge 2$
\begin{equation}
\label{fe32}
\limsup_{N\to\infty}\frac{D_{multi}(N,p,\nu_{K})}{N^{\gamma_{p,\delta,s,\epsilon}}}<\infty
\end{equation}
 where
$$\gamma_{p,\delta,s,\epsilon}=\epsilon(s-1)+2^{-s}+\kappa_p(\gamma_p+\delta)(\frac{1-(1-\kappa_p)^{s-1}}{\kappa_p}-2^{-s+1}\frac{1-(2(1-\kappa_p))^{s-1}}{2\kappa_p-1})+$$$$+\kappa_p 2^{-s}(1-\frac2p)\frac{1-(2(1-\kappa_p))^{s-1}}{2\kappa_p-1}+O_p((1-\kappa_p)^s).$$

We will show now that if $p>8$ then
$$\gamma_p\le \frac{2\kappa_p-1}{2\kappa_p}+\frac12-\frac1p=\frac{p-8}{2p-10}+\frac12-\frac1p.$$
If we manage to do this, it will suffice to let $p\to 8$ to get $\gamma_8\le \frac38$, hence actually $\gamma_8=\frac38$, as desired.

We first note that if $p>8$
\begin{equation}
\label{fe24}
2(1-\kappa_p)=\frac{6}{p-2}<1.
\end{equation}
Assume for contradiction that for some $p>8$ we have
\begin{equation}
\label{fe26}
\gamma_p>\frac{2\kappa_p-1}{2\kappa_p}+\frac12-\frac1p.
\end{equation}
A simple computation using \eqref{fe24} and \eqref{fe26} shows that for $s$ large enough, and $\epsilon,\delta$ small enough we have
\begin{equation}
\label{fe33}
\gamma_{p,\delta,s,\epsilon}<\gamma_p
\end{equation}
and
$$(1+\epsilon)(\frac12-\frac1p)<1-\frac5p.$$
Fix such  $\epsilon,\delta, s$ and choose now $K$ so large that
\begin{equation}
\label{fe35}
(1+\epsilon)(\frac12-\frac1p)+\beta(K,p,\epsilon)<1-\frac5p,
\end{equation}
where $\beta(K,p,\epsilon)$ is from Theorem \ref{ft2}.

Now, \eqref{fe31} combined with \eqref{fe35} and \eqref{fe30} shows that for $N\ge K$
\begin{equation}
\label{fek16}
D(N,p)\lesssim_{K,p,\epsilon}\log_2 N\max_{1\le M\le N}(\frac{M}{N})^{(1+\epsilon)(\frac1p-\frac12)}D_{multi}(M,p,\nu_{K}).
\end{equation}
We have two possibilities.

First, if $\gamma_{p,\delta,s,\epsilon}<(1+\epsilon)(\frac12-\frac1p)$ then using \eqref{fek16} and \eqref{fe32} we can write
$$D(N,p)\lesssim_{K,p,\epsilon}\log_2 N\max_{1\le M\le N}(\frac{M}{N})^{(1+\epsilon)(\frac1p-\frac12)}M^{(1+\epsilon)(\frac12-\frac1p)}=\log_2 NN^{(1+\epsilon)(\frac12-\frac1p)}.$$
This contradicts the combination of \eqref{fe30} and \eqref{fe35}.

Second, if $\gamma_{p,\delta,s,\epsilon}\ge (1+\epsilon)(\frac12-\frac1p)$ then using \eqref{fek16} again we can write
$$D(N,p)\lesssim_{K,p,\epsilon}\log N\max_{1\le M\le N}(\frac{M}{N})^{(1+\epsilon)(\frac1p-\frac12)}M^{\gamma_{p,\delta,s,\epsilon}}$$$$\lesssim_{K,p,\epsilon}\log NN^{\gamma_{p,\delta,s,\epsilon}},$$
which contradicts \eqref{fe33} and the definition of $\gamma_p$.
In conclusion, inequality \eqref{fe26} can not hold, and the proof of Theorem \ref{tfek6} is complete.

\end{document}